\documentclass{article}
\usepackage{amsmath}
\usepackage{latexsym}
\usepackage{amssymb}
\usepackage{colonequals}
\newtheorem{thm}{Theorem}[section]
\newtheorem{la}[thm]{Lemma}
\newtheorem{Defn}[thm]{Definition}
\newtheorem{Remark}[thm]{Remark}
\newtheorem{Conj}[thm]{Conjecture}
\newtheorem{prop}[thm]{Proposition}

\newtheorem{Example}[thm]{Example}
\newtheorem{Number}[thm]{\!\!}
\newenvironment{defn}{\begin{Defn}\rm}{\end{Defn}}
\newenvironment{example}{\begin{Example}\rm}{\end{Example}}
\newenvironment{rem}{\begin{Remark}\rm}{\end{Remark}}

\newenvironment{numba}{\begin{Number}\rm}{\end{Number}}
\newenvironment{proof}{{\noindent\bf Proof.}}%
                  {\nopagebreak\hspace*{\fill}$\Box$\medskip\par}

\newcommand{\cs}{{\mathfrak s}}
\newcommand{\cW}{{\mathcal W}}
\newcommand{\cH}{{\mathcal H}}
\newcommand{\cM}{{\mathcal M}}

\newcommand{\cO}{{\mathcal O}}

\newcommand{\cL}{{\mathcal L}}

\newcommand{\cA}{{\mathcal A}}

\newcommand{\cB}{{\mathcal B}}

\newcommand{\R}{{\mathbb R}}
\newcommand{\C}{{\mathbb C}}
\newcommand{\bD}{{\mathbb D}}
\newcommand{\K}{{\mathbb K}}
\newcommand{\N}{{\mathbb N}}
\newcommand{\Z}{{\mathbb Z}}
\newcommand{\Sph}{{\mathbb S}}
\newcommand{\mto}{\mapsto}
\newcommand{\sub}{\subseteq}
\newcommand{\wt}{\widetilde}
\newcommand{\wh}{\widehat}
\newcommand{\wb}{\overline}
\DeclareMathOperator{\id}{id}

\DeclareMathOperator{\GL}{GL}

\DeclareMathOperator{\pr}{pr}

\DeclareMathOperator{\ev}{ev}

\DeclareMathOperator{\im}{im}

\DeclareMathOperator{\op}{op}

\DeclareMathOperator{\Repart}{Re}
\DeclareMathOperator{\ind}{ind}
\DeclareMathOperator{\res}{res}
\DeclareMathOperator{\diag}{diag}
\DeclareMathOperator{\EXP}{EXP}

\newcommand{\dl}{{\displaystyle \lim_{\longrightarrow}}}
\newcommand{\pl}{{\displaystyle \lim_{\longleftarrow}}}
\newcommand{\cg}{{\mathfrak g}}

\DeclareMathOperator{\Lip}{Lip}
\begin{document}
%
%
\begin{center}
{\bf\Large Birkhoff decompositions
for loop\vspace{2.3mm} groups with coefficient algebras}\\[6mm]
{\bf Helge Gl\"{o}ckner}\vspace{2mm}
\end{center}
\begin{abstract}
\noindent
For certain locally convex topological
algebras $\cA\sub C(\Sph,\C)$ generalizing
the decomposable $R$-algebras of Gohberg
and Fel'dman,
we establish Birkhoff decompositions
in $\GL_n(\cA)$ of the form
$f=f_+Df_-$, where $D(z)$
a diagonal metrix with entries $z^{\kappa_1},\ldots,z^{\kappa_n}$,
and moreover $f_+ \in \GL_n(\cA)$ is the boundary values of a holomorphic
$\GL_n(\C)$-valued function on the open unit disk
and $f_-\in \GL_n(\cA)$ the boundary
value of a $\GL_n(\C)$-valued holomorphic function
on the complement of the closed unit disk in the Riemann sphere.
Corresponding to the case $\kappa_1=\cdots=\kappa_n=0$,
analogous decompositions are available
in Lie groups $G_\cA\leq C(\Sph,G)$
with Lie algebra $\cg\otimes_\C\cA$,
for $G$ a finite-dimensional complex Lie group
on which holomorphic functions separate points.
Birkhoff decompositions in $\cO(\C^\times ,G)$
are also obtained. Some of the results
remain valid of $G$ is a complex Banach-Lie group
on which holomorphic functions
separate points.\vspace{2.4mm}
\end{abstract}
%
%
%
%
\section{Introduction and statement of the results}
The book ``Loop Groups''
by Pressley and Segal~\cite{PaS}
is a classical source on
Birkhoff decompositions
for loops in $\GL_n(\C)$
or in complexifications
$K_\C\leq \GL_n(\C)$
of compact real Lie groups~$K$
(a topic which started in \cite{Bi1}).
As far as $\GL_n(\C)$ is
concerned, their results include (see
Remark~\ref{complexification}
for $K_\C$):
\begin{numba}\label{Pressley}
(Pressley-Segal 1986)
Each $f\in C^\infty(\Sph,\GL_n(\C))$ can be written as
\[
f=f_+ D\hspace*{.5mm}f_-
\]
where $D\colon \Sph\to \Sph^n$ is a continuous group
homomorphism, $f_+\in C^\infty(\Sph,\GL_n(\C))$ admits a continuous
extension
$f_+^< \colon \wb{\bD}\to\GL_n(\C)$ which is holomorphic
on~$\bD$, and $f_-\in C^\infty(\Sph,\GL_n(\C))$
admits a continuous extension $f_-^>\colon
\hat{\C}\setminus \bD\to\GL_n(\C)$
which is holomorphic on $\hat{\C}\setminus \wb{\bD}$.
The set $C^\infty(\Sph,\GL_n(\C))^+$
of such $f_+$ and the set $C^\infty(\Sph, \GL_n(\C))^\ominus$
of such $f_-$ with $f_-^>(\infty)={\bf 1}$
are embedded submanifolds and Lie subgroups
of the complex Lie group $C^\infty(\Sph,\GL_n(\C))$.
The product map
\[
C^\infty(\Sph,\GL_n(\C))^+\times C^\infty(\Sph,\GL_n(\C))^\ominus
\to C^\infty(\Sph,\GL_n(\C))_0,\;\; (f_+,f_\ominus)\mto f_+f_\ominus
\]
has dense open image and is a $\C$-analytic diffeomorphism
onto the latter.
\end{numba}
Here $\Sph:=\{z\in\C\colon |z|=1\}$
is the circle group, $\wh{\C}:=\C\cup\{\infty\}$
the Riemann sphere and $\bD:=\{z\in\C\colon |z|<1\}$
the open unit disk.
Moreover, ${\bf 1}\in \GL_n(\C)$
is the unit matrix and $C^\infty(\Sph,\GL_n(\C))_0$
the connected component of the neutral element
$z\mto{\bf 1}$ in the Lie group
$C^\infty(\Sph,\GL_n(\C))$ of smooth loops in $\GL_n(\C)$.
\begin{rem}\label{complexification}
(a) Pressley and Segal~\cite{PaS}
get analogues of \ref{Pressley}
also for $C^\omega$ in place of $C^\infty$.\\[1.5mm]
(b) They also show $C^\infty(\Sph,K_\C)^+\times
C^\infty(\Sph,K_\C)^\ominus\to C^\infty(\Sph,K_\C)_0$
is $\C$-analytic diffeomorphism onto a dense open subset.
\end{rem}
The approach of Pressley and Segal is geometric in nature
and fairly complicated. The effort is justified
as it provides valuable additional
information beyond the Birkhoff decompositions,
like stratifications and cell decompositions.\\[2.3mm]
In this article, we propose a complementary
approach to Birkhoff decompositions,
based on two pillars:
\begin{itemize}
\item
The approach of Goberg-Fel'dman \cite{GaF}
to
Birkhoff decompositions for loops in $\GL_n(\C)$
via Gelfand theory of
commutative Banach algebras;
\item
Techniques from infinite-dimensional
differential calculus/holomorphy
and infinite-dimensional Lie theory.
\end{itemize}
More precisely, we shall generalize
the approach by Goberg and Fel'dman,
replacing Banach algebras
with certain locally convex topological
algebras, the continuous inverse algebras.
In this article,
the algebra multiplication
\[
\cA\times \cA\to\cA, \quad (f,g)\mto fg
\]
of a topological algebra is always assumed
jointly continuous.
The following concept
goes back to work of Waelbroeck
(cf.\ \cite{Wa1});
see \cite{CIA} for a recent account with
a view towards infinite-dimensional Lie theory.
\begin{numba}
A locally convex, Hausdorff, associative
complex topological algebra~$\cA$ with unit
is called a \emph{continuous inverse algebra}
(or cia) if the group $\cA^\times$
of invertibe elements is open in~$\cA$
and the inversion map $\cA^\times \to \cA$, $a\mto a^{-1}$
is continuous.
\end{numba}
We recall from Gohber-Fel'dman \cite{GaF}:
\begin{numba}
A Banach algebra $\cA\sub C(\Sph,\C)$
is called an \emph{$R$-algebra}
if the set of restrictions $f|_\Sph$
of all rational functions $f$ with poles off $\Sph$
is dense in $\cA$.
\end{numba}
We shall us the following generalization
beyond Banach algebras.
\begin{defn}
A cia $\cA\sub C(\Sph,\C)$
is called \emph{of type $R$}
(or an $R$-cia) if the set of restrictions $f|_\Sph$
of all rational functions~$f$
with poles off $\Sph$ is dense in $\cA$.
\end{defn}
Then the inclusion map
$\cA\to (C(\Sph,\C),\|\cdot\|_\infty)$ is continuous,
moreover,
$\cA^\times=\cA\cap C(\Sph,\C)^\times$ holds
(so-called isospectrality),
see Lemma~\ref{basicR}.\\[2mm]
Mimicking Gohberg-Fel'dman,
we make the following definitions
if $\cA$ is an $R$-cia:
\begin{defn}
We write $\cA^+\sub\cA$
for the closure of $\C[z]$ in $\cA$
and $\cA^\ominus\sub\cA$ for
the closure of $z^{-1}\C[z^{-1}]$
in $\cA$.
An $R$-cia $\cA$ is called \emph{decomposing}
if $\cA=\cA^+\oplus\cA^\ominus$ as a topological
vector space.
\end{defn}
\begin{numba}
In the following, we write
\begin{equation}\label{fou-koff}
\wh{f}_k:=\int_0^1 e^{2\pi i k t} f(e^{2\pi i t})\, dt
\end{equation}
for the $k$th Fourier coefficient of $f\in C(\Sph,\C)$,
for $k\in\Z$.
\end{numba}
As in \cite{GaF}, one shows:
%
%
\begin{numba}\label{char-deco}
If $\cA\sub C(\Sph,\C)$ is a decomposing $R$-cia
and $f\in \cA$, then the following properties
are equivalent:
\begin{itemize}
\item[(a)]
$f\in\cA^+$
\item[(b)]
$\hat{f}_k=0$ for all $k< 0$;
\item[(c)]
$f$ has continuous extension
$f^<\colon \wb{\bD}\to\C$ which is holomorphic on $\bD$.
\end{itemize}
Likewise, the following conditions
are equivalent:
\begin{itemize}
\item[(d)]
$f\in\cA^\ominus$;
\item[(e)]
$\hat{f}_k=0$ for all $k\geq  0$;
\item[(f)]
$f$ has a continuous extension
$f^> \colon \wh{\C}\setminus \bD\to\C$ which is
holomorphic on $\wh{\C}\setminus \overline{\bD}$,
with $f^>(\infty)=0$.
\end{itemize}
We set $\cA^-:=\cA^\ominus + \C {\bf 1}$
(where ${\bf 1}$ is the constant function
$z\mto 1$). Then $\cA^-$ is the closure of $\C[z^{-1}]$
in~$\cA$ and $f\in\cA$ is in $\cA^-$
if and only if $\hat{f}_k=0$ for all $k>0$.
\end{numba}
\begin{numba}
If $\cA\sub C(\Sph,\C)$ is
an algebra which may not be a decomposing
$R$-cia, we define
$\cA^+$ as the set of all $f\in \cA$
satisfying condition~(c) of \ref{char-deco},
and $\cA^\ominus$
as the set of all $f\in\cA$ satisfying condition~(f)
of \ref{char-deco}.
Then $\cA^+$ is a unital subalgebra
of $\cA$ and $\cA^\ominus$
is a non-unital subalgebra of~$\cA$.
\end{numba}
The following permanence
properties of the class of $R$-cias
are easily checked.
\begin{prop}\label{permafrost}
\begin{itemize}
\item[\rm(a)]
If $\cA_1\supseteq \cA_2\supseteq\cdots$
are $R$-cias and all inclusion maps
are continuous unital algebra homomorphisms,
then
$\bigcap_{n\in\N}\cA_n=\pl\, \cA_n$\vspace{-1mm}
is an $R$-cia which is decomposing if each $\cA_n$ is so.
\item[\rm(b)]
If\vspace{-3mm}
\[
\cA_1\sub \cA_2\sub\cdots
\]
are $R$-algebras and all inclusion maps
are continuous unital algebra homomorphisms,
then
the locally convex direct limit $\cA=\bigcup_{n\in\N}\cA_n$
is an $R$-cia which is decomposing if each $\cA_n$ is so.
\end{itemize}
\end{prop}
Note that, in the situation of~(b),
the inclusion map $\iota: \cA\to C(\Sph,\C)$
is continuous on each $\cA_n$
and hence continuous on the direct limit locally convex space~$\cA$, being linear.
Since $C(\Sph,\C)$ is Hausdorff, so is $\cA$.
Since all of the inclusions $\cA_n\to C(\Sph,\C)$
are isospectral,
\[
\cA^\times =\iota^{-1}(C(\Sph,\C)^\times)
\]
follows, whence $\cA^\times$ is open in $\cA$.
By \cite{AaN} or \cite{DaW}, $\cA$
is locally $m$-convex in the sense of Michael~\cite{mcx},
entailing that the inversion map $\cA^\times \to \cA$
and the algebra multiplication
$\cA\times\cA\to\cA$ are continuous.
Hence $\cA$ is a cia.
Since $\C[z,z^{-1}]$
is dense in each $\cA_n$, it is dense
in~$\cA$.
All rational functions with poles off $\Sph$
are in $\cA_1$ and hence in~$\cA$,
whence $\cA$ is an $R$-cia.
\begin{rem}
If $\cA_1\sub \cA_2\sub\cdots$
are merely $R$-cias in Proposition~\ref{permafrost}\,(b),
then the conclusions still hold if we assume,
moreover, that the algebra multiplication
$\cA\times\cA\to\cA$ is continuous.
In fact, $\cA$ is locally $m$-convex in this case by \cite{HaW},
and now we can complete the argument
as before.
\end{rem} 
\begin{numba}
Let us compile some
topological algebras $\cA$ which may
be interesting algebras of
coefficients for loop groups,
and their properties.
All of them are subalgebras of $C(\Sph,\C)$
and contain the algebra
$\C[z,z^{-1}]$ of Laurent polynomials as
a dense subalgebra,
\[
\C[z,z^{-1}]\sub \cA\sub C(\Sph,\C).
\]
\begin{itemize}
\item[(a)]
The Banach algebra
$(C(\Sph,\C),\|\cdot\|_\infty)$
of continuous complex-valued functions on~$\Sph$
is an $R$-algebra.
It is a classical fact thet $C(\Sph,\C)$
is not decomposing (see \cite{GaF}).
\item[(b)]
The Wiener algebra
$\cW :=\{f\in C(\Sph,\C)\colon (\hat{f}_k)_{k\in\Z}
\in \ell^1\}$ is a decomposing $R$-algebra
using the norm given by $\|f\|_\cW:=\|(\wh{f}_k)_{k\in\Z}\|_{\ell^1}=\sum_{k\in \Z}
|\wh{f}_k|$ (see \cite{GaF}).
Here
\begin{equation}\label{simply-plus}
f^<(z):=\sum_{k=0}^\infty \wh{f}_k z^k\quad\mbox{for $z\in \wb{\bD}$}
\end{equation}
for $f\in \cW^+$ and
\begin{equation}\label{simply-minus}
f^>(z):=\sum_{k=1}^\infty\wh{f}_{-k}z^{-k}\quad\mbox{for $z\in \C\setminus \bD$}
\end{equation}
for $f\in \cW^\ominus$, with $f^>(\infty):=0$.
\item[(c)]
For $m\in[0,\infty[$,
the set
\[
\cW(m) :=\left\{f\in C(\Sph,\C)\colon \sum_{k\in\Z} |k|^m|\hat{f}_k|<\infty\right\}
\]
is a unital subalgebra of $\cW$.
The norm $f\mto\sum_{k\in\Z}\max\{|k|^m,1\}|\hat{f}_k|$
turns $\cW(m)$ into a Banach space and makes
the algebra multiplication continuous.
Let $L_f(g):=fg$ for $f,g\in\cW(m)$.
Then $\|f\|_{\cW(m)}:=\|L_f\|_{\op}$
defines an equivalent norm on $\cW(m)$ which is submultiplicative and turns
$\cW(m)$ into an $R$-algebra which is decomposing.
%
%
We can think of the weighted Wiener algebras
as Sobolev-like spaces.
\item[(d)]
The Fr\'{e}chet algebra $C^\infty(\Sph,\C)$
of all smooth complex-valued functions on the circle
can be regarded as the projective limit
\[
C^\infty(\Sph,\C)\,=\,
\bigcap_{m\in\N_0}\cW(m)\, =\, \pl \, \cW(m)
\]
as a locally convex space, as the intersection
of the weighted $\ell^1$-spaces corresponding to $\cW(m)$
is the Fr\'{e}chet space $\cs$ of rapidly decreasing
two-sided sequences, which equals the set of
Fourier transforms $(\wh{f}_k)_{k\in\Z}$
of smooth functions $f\in C^\infty(\Sph,\C)$.
By Proposition~\ref{permafrost}\,(a),
$C^\infty(\Sph,\C)$ is a decomposing $R$-cia.
\item[(e)]
The Silva space
$C^\omega(\Sph,\C)=\dl \,\cO_b(A_n,\C)$
of real-analytic $\C$-valued functions
is a decomposing $R$-cia.
Here $\cO_b(\cA_n,\C)$ denotes the Banach algebra
of bounded holomorphic functions on the open annulus
\[
A_n:=\left\{z\in \C\colon \,1-\frac{1}{n} <|z|<1+\frac{1}{n}\right\},
\]
endowed with the supremum norm.
Actually, it will be useful later to use
the following equivalent norm
which also makes $\cO_b(A_n,\C)$ a Banach algebra:
\[
\|f\|_n:=\max\{\|f\|_\infty,\|f|_\Sph\|_\cW\}
\]
where $\|\cdot\|_\cW$ is the norm of the Wiener algebra.\\[2.3mm]
To establish the asserted properties, note that
$C^\omega(\Sph,\C)$
is locally $m$-convex by \cite{AaN,DaW}
and has an open unit group,
as the inclusion map $\iota\colon C^\omega(\Sph,\C)\to C(\Sph,\C)$
is continuous and linear
and
\[
\iota^{-1}(C(\Sph,\C)^\times)=C^\omega(\Sph,\C),
\]
using that $1/f$ is real analytic for each
real-analytic function $f\colon \Sph\to\C$
such that $0\not\in f(\Sph)$.
The principal and regular parts
of the Laurent series of $f\in \cO_b(A_n,\C)$
converge uniformly on $A_{n+1}$,
entailing that $\C[z,z^{-1}]$
(and hence also rational functions with poles off $\Sph$)
are dense in $C^\omega(\Sph,\C)$.
\item[(f)]
Finally, we shall consider the Fr\'{e}chet
algebra
$\cO(\C^\times,\C)$ of all
holomorphic functions
on the punctured plane, endowed
with the topology of compact convergence.
This algebra is not a cia
as its unit group is not open.
Still, it is locally $m$-convex and
\begin{equation}\label{addi}
\cO(\C^\times ,\C)=\cO(\C^\times,\C)^+\oplus \cO(\C^\times,\C)^\ominus
\end{equation}
holds as a locally convex space.
In fact, as the decomposition of a holomorphic
function
into its regular part and principal part is unique,
the map in (\ref{addi}) is bijective.
As its inverse map is continuous
and both sides are Fr\'{e}chet spaces,
the map is an isomorphism of topological
vector spaces by the open mapping theorem.
\end{itemize}
\end{numba}
%
%
%
Concerning loops in $\GL_n(\C)$,
we obtain the following
results,
whose special cases for $R$-algebras
can be found in~\cite{GaF}.
Recall that a locally convex topological
vector space is called \emph{Mackey complete}
if every Mackey-Cauchy sequence
converges in it (see \cite{KaM}
for an in-depth discussion of this property).
Every sequentially complete
locally convex space (in which every Cauchy sequence
converges) is Mackey complete.
\begin{prop}\label{birkho-1}
If $\cA$ is a decomposing $R$-cia
which is Mackey complete,
then every $g\in \cA^\times$ can be written as
\[
g=g_+(\id_\Sph)^\kappa g_-\vspace{-1mm}
\]
with $g_\pm\in (\cA^\pm)^\times$ and some $\kappa\in \Z$.
Moreover, $\kappa=\ind_0(g)$
is the winding number of~$g$
with respect to the origin $0\in \C$.
\end{prop}
\begin{prop}\label{birkho-n}
Let $\cA$ be an $R$-algebra.
Then each $g\in \GL_n(\cA)$ can be written in the form
\[
g=g_+ D \, g_-
\]
where $g_+\in \GL_n(\cA^+)$,
$g_-\in \GL_n(\cA^-)$ and
$D(z)=\diag(z^{\kappa_1},\ldots, z^{\kappa_n})$
with integers
$\kappa_1\geq\kappa_2\geq\cdots\geq \kappa_n$.
The $n$-tuple $(\kappa_1,\ldots,\kappa_n)$
of ``partial indices'' is unique.
The same conclusion holds
if $\cA$ is any Mackey complete, decomposing $R$-cia
and $\GL_n(\cA_+)\GL_n(\cA_-)$ is an identity neighbourhood
in~$\GL_n(\cA)$.
\end{prop}
\begin{numba}
If $\cA$ is a complex locally convex commutative topological
algebra and $\cg$ a finite-dimensional
complex Lie algebra,
we endow
\[
\cg_\cA:=\cg\otimes_\C \cA
\]
with the projective tensor topology
(as in \ref{proj-tensor}).
If $\cA$ is a subalgebra
of $C(\Sph,\C)$, we have
\[
\cg_\cA =\cg\otimes_\C \cA\sub \cg\otimes_\C C(\Sph,\C)
=C(\Sph,\cg)
\]
and we consider $\cg_A$ as a subalgebra of
$C(\Sph,\cg)$. If the inclusion map
$\cA\to C(\Sph,\C)$ is continuous,
then also the inclusion map
$\cg_\cA\to C(\Sph,\cg)$.
\end{numba}
\begin{numba}\label{previous}
If $G$ is a complex Lie group
modeled on a Banach space
(a complex Banach--Lie group),
with Lie algebra $\cg:=T_eG$
and exponential function
$\exp_G\colon \cg\to G$,
then $C(\Sph,G)$ can be made
a complex Banach--Lie group with Lie algebra
$C(\Sph,\cg)$, as is well known.
Its exponential function is the map
\[
C(\Sph,\cg)\to C(\Sph,G),\quad f\mto \exp_G\circ f.
\]
If $\cA$ is an $R$-algebra,
then the inclusion map $\cg_\cA\to C(\Sph,\cg)$
is an injective continuous Lie algebra homomorphism
between Banach--Lie algebras.
Hence
\[
G_\cA:=\langle \exp_G\circ f\colon f\in \cA\rangle
\]
can be made a Banach--Lie group with Lie algebra $\cg_\cA$,
see \cite{vEK}.
If $\cA$ is decomposing, we get
analogous subgroups
$G_\cA^+$ and $G_\cA^\ominus$ in $C(\Sph,G)$
which are Banach-Lie groups with Lie algebras
$\cg_\cA^+:=\cg_{\cA^+}$ and $\cg_\cA^-:=\cg_{\cA^\ominus}$, respectively.
The inclusion maps $G_\cA^+\to G_\cA$
and $G_\cA^\ominus\to G_\cA$
are complex analytic group homomorphisms.
\end{numba}
\begin{numba}
If $\cA$ is a Mackey complete $R$-cia,
we get a Lie group
$G_\cA\sub C(\Sph,G)$ with Lie algebra $\cg_\cA$
and Lie groups $G_\cA^+$ and $G_\cA^\ominus$
which are subgroups of $C(\Sph,G)$
if $\cA$ is, moreover, decomposing.
The latter have Lie algebras $\cg_\cA^\oplus$
and $\cg_\cA^\ominus$, respectively
and properties as described in~\ref{previous}
(see \ref{details-BCH}).
\end{numba}
Note that $G_\cA$
is the connected component $C^\infty(\Sph,G)_0$
of $C^\infty(\Sph,G)$
if $\cA=C^\infty(\Sph,\C)$.
Likewise,
$G_\cA=C^\omega(\Sph,G)_0$
if $\cA=C^\omega(\Sph,\C)$.
\begin{numba}
If $\cg$ is a complex Banach--Lie algebra with norm $\|\cdot\|$,
we can consider the Fr\'{e}chet--Lie algebra
\[
C^\infty(\Sph,\cg)
\]
of smooth $\cg$-valued functions on~$\Sph$,
endowed with the compact-open $C^\infty$-topology
(as in \cite{GaN}).
We can also consider
the locally convex topological Lie algebra
\[
C^\omega(\Sph,\cg)=\dl \cO_b(A_n,\cg)
\]
of real-analytic $\cg$-valued functions on~$\Sph$,
the Wiener--Lie algebra
\[
\cg_\cW:=\{f\in C(\Sph,\cg)\colon (\wh{f}_k)_{k\in\Z}\in \ell^1(\Z,\cg)\}
\]
and the weighted Wiener--Lie algebras
\[
\cg_{\cW(m)}:=\left\{
f\in C(\Sph,\cg)\colon \sum_{k\in \Z}|k|^m\|\hat{f}_k\|<\infty\right\}
\]
for $m\in [0,\infty[$,
where $\hat{f}_k$ is defined as the $\cg$-valued
integral~(\ref{fou-koff}).
The Wiener--Lie algebras and weighted
Wiener--Lie algebras
are Banach--Lie algebras.
If $\cg=T_eG$
for some complex Banach--Lie group~$G$,
we define
\[
G_\cW,\; G_\cW^+,\; G_\cW^\ominus,\;
G_{\cW(m)},\; G_{\cW(m)}^+,\;\mbox{and}\;
G_{\cW(m)}^\ominus
\]
as in~\ref{previous}.
\end{numba}
The inverse function theorem
readily implies the following
(see Section~\ref{sec-G}).
\begin{prop}\label{deco-easy}
Let $\cA$ be a decomposing $R$-algebra
and $G$ be a finite-dimensional
complex Lie group,
or
$\cA=\cW(m)$ with $m\in [0,\infty[$
and $G$ be a complex
Banach--Lie group.
Assume that $\cO(G,\C)$ separates points
on~$G$.
Then
$G_\cA^+G_\cA^\ominus$ is open in $G_\cA$ and the
product map
\[
\pi\colon G_\cA^+\times G_\cA^\ominus\to G_\cA^+ G_\cA^\ominus
\]
is a $\C$-analytic diffeomorphism.
\end{prop}
For a complex Banach--Lie group~$G$,
we shall use the complex Fr\'{e}chet--Lie groups
$\cO(\C^\times, G)$, $\cO(\C,G)$,
and $\cO(\wh{\C}\setminus \{0\},G)$
of all holomorphic $G$-valued functions
on $\C^\times$, $\C$, and $\wh{\C}\setminus \{0\}$,
respectively,
as constructed in~\cite{NaW}. We recall from \cite{NaW}:
\begin{numba}
Let $M$ be a non-compact Riemann surface
with finitely generated fundamental
group.
Let $G$ be a complex Banach--Lie group
with neutral element~$e$, Lie algebra $\cg:=T_eG$
and exponential function $\exp_G\colon \cg\to G$.
Then the group $\cO(M,G)$
can made a Fr\'{e}chet--Lie group
with Lie algebra $\cO(M,\cg)$
and exponential function
\[
\cO(M,\cg)\to \cO(M,G),\quad f\mto\exp_G\circ f,
\]
such that the following
\emph{exponential law} holds:
For each complex analatic manifold~$N$
modelled on a complex locally convex space,
a map
\[
f\colon N\to \cO(M,G)
\]
is complex analytic if and only if
the corresponding mapping
\[
f^\wedge\colon N\times M\to G,\quad (n,m)\mto f(n)(m)
\]
is complex analytic.
\end{numba}
\begin{rem}\label{exp-sub}
By construction,
$\cO(\wh{\C}\setminus \{0\},G)_*:=\{f\in \cO(\wh{\C}\setminus\{0\}\colon
f(\infty)=e\}$
is an embedded complex
submanifold of $\cO(\wh{\C}\setminus\{0\},G)$
and a Lie subgroup.
As a consequence, the exponential
law also holds for functions $f\colon N\to \cO(\wh{\C}\setminus
\{0\},G)_*$.
\end{rem}
The following three theorems are the main results of the article.
\begin{thm}\label{main-1}
For each complex Banach--Lie group $G$
on which $\cO(G,\C)$ separates points, the
map
\[
p \colon \cO(\C,G)\times \cO(\hat{\C}\setminus \{0\},G)_*\! \to
\cO(\C^\times,G),\quad\!\!\!
(g_+,g_\ominus)\mto g_+|_{\C^\times}g_\ominus|_{\C^\times}
\]
has open image and is a $\C$-analytic diffeomorphism onto its image.
\end{thm}
\begin{numba}
Given a complex Banach--Lie group
$G$ on $\cO(G,\C)$ separates points,
we let $C^\infty(\Sph,G)^+$
be the set of all $f\in C^\infty(\Sph,G)$
admitting a continuous extension
$f^<\colon \wb{\bD}\to G$
which is complex analytic on~$\bD$.
We let $C^\infty(\Sph,G)^\ominus$
be the set of all $f\in C^\infty(\Sph,G)$
admitting a continuous extension
$f^>\colon \wh{\C}\setminus \bD\to G$
which is complex analytic on $\wh{\C}\setminus\wb{\bD}$
and satisfies $f(\infty)=e$.
Then $C^\infty(\Sph,G)^+$
and $C^\infty(\Sph,G)^\ominus$
are subgroups of $C^\infty(\Sph,G)$.
Subgroups $C^\omega(\Sph,G)^+$
and $C^\omega(\Sph,G)^\ominus$
of $C^\omega(\Sph,G)$
can be defined along the same lines.
\end{numba}
\begin{thm}\label{main-2}
For each complex Banach--Lie group $G$
on which $\cO(G,\C)$ separates points, the
map
\[
\pi\colon C^\omega(\Sph,G)^+\times C^\omega(\Sph,G)^\ominus
\to C^\omega(\Sph,G).
\]
has open image and is a $\C$-analytic diffeomorphism onto its image.
\end{thm}
In the following theorem,
Lie subgroups are understood
as embedded complex submanifolds.
\begin{thm}\label{main-3}
Let $G$ be a finite-dimensional complex Lie group
which is
a Lie subgroup of $\GL_n(\C)$
for some $n\in\N$. Then the map
\[
\pi\colon C^\infty(\Sph,G)^+\times C^\infty(\Sph,G)^\ominus
\to C^\infty(\Sph,G).
\]
has open image and is a $\C$-analytic diffeomorphism onto its image.
\end{thm}
Let us compare our approach with \cite{PaS}.
\begin{rem}
I our approach, we recover all of the known facts from
\ref{Pressley} and Remark~\ref{complexification},
except for the density property.
The following aspects are new in our approach:
\begin{itemize}
\item[(a)]
The product map (in the Birkhoff decompositions)
is a $\C$-analytic diffeomorphism
for $\cO(\C^\times, G)$ and $C^\omega(\Sph,G)$;
\item[(b)]
The theory applies to more general complex Lie groups
$G$ (including some Banach--Lie groups), not only to
complexifications $K_\C$;
\item[(c)]
The approach covers further classes of loop groups
based on other coefficient algebras.
\end{itemize}
\end{rem}
We mention that our work has points of contact with
\cite{GaW}, but differs in details.
A broader discussion will be given in a later
version of this manuscript.

\section{Preliminaries and basic facts}\label{secprels}
We write $\N:=\{1,2,\ldots\}$
and $\N_0:=\N\cup\{0\}$.
All topological vector spaces are assumed Hausdorff.
If $(E,\|\cdot\|)$
is a normed space, we write $B^E_r(x):=\{y\in E\colon \|y-x\|<r\}$
and $\wb{B}^E_r(x):=\{y\in E\colon \|y-x\|\leq r\}$
for $x\in E$ and $r>0$.\\[2.3mm]
See \cite{Res, GaN} for the following concept
(as well as \cite{BaS,Mil} if $F$
is sequentially complete).
\begin{defn}
Let $E$ and $F$ be locally convex
complex topological vector spaces
and $U\sub E$ be an open subset.
A map $f\colon U\to F$
is called \emph{complex analytic}
(or $\C$-analytic)
if it is continuous
and, for each $x\in U$,
\[
f(y)=\sum_{k=0}^\infty p_k(y-x)
\]
for all $y$ in some $x$-neighbourhood in~$U$,
where $p_k\colon E\to F$
is a continuous complex homogeneous polynomial
of degree~$k$.
\end{defn}
If $F$ is Mackey complete, then a continuous function $f$ as before
is complex analytic if and only if $f$ is \emph{holomorphic}
in the sense that the complex directional derivative
\[
df(x,y):=\lim_{z\to 0}\frac{f(x+zy)-f(x)}{z}
\]
(with $z\in\C\setminus\{0\}$ in some $0$-neighbourhood)
exists for all $x\in U$ and $y\in E$, and $df\colon U\times E\to F$
is continuous (see, e.g., \cite{GaN}).
We shall then use the words ``complex analytic''
and ``holomorphic''
interchangeably.\\[2.3mm]
Following \cite{Res,GaN}
(see already \cite{Mil}
for sequentially complete spaces),
we define:
\begin{defn}
Let $E$ and $F$ be locally convex
real topological vector spaces
and $U\sub E$ be an open subset.
A map $f\colon U\to F$
is called \emph{real analytic}
if it admits a complex-analytic
extension $V\to F_\C$
with values in the a complexification $F_\C$ of~$F$,
defined on an open subset $V\sub E_\C$ with $U\sub V$.
\end{defn}
Complex-analytic maps are real analytic (see \cite{GaN}).
Compositions of composable $\K$-analytic maps
are $\K$-analytic for $\K\in\{\R,\C\}$ (see \cite{Res,GaN}).
Therefore complex manifolds
and complex Lie groups modelled
on complex locally convex spaces can be defined
as expected.
\begin{numba}
A complex-analytic manifold (or simply: complex manifold)
modelled on a complex locally convex space~$E$
is a Hausdorff topological space~$M$,
together with a maximal
set $\cA$ of homeomorphisms
$\phi\colon U_\phi\to V_\phi$ (``charts'')
from an open subset $U_\phi\sub M$
onto an open subset $V_\phi\sub E$
such that $\bigcup_{\phi\in\cA}U_\phi=M$
and the chart changes $\phi\circ \psi^{-1}$
are complex analytic for all $\phi,\psi\in \cA$.
\end{numba}
Complex-analytic maps between complex manifolds
are defined as usual, as continuous
maps which are complex analytic in charts
(see \cite{GaN} for details).
\begin{numba}
A \emph{complex Lie group}
modelled on a complex locally convex space~$E$
is a group~$G$, endowed with
a complex manifold structure modelled on~$E$
which turns the group operations into complex analytic
maps (see \cite{GaN} for more details).
\end{numba}
When we use the word ``complex manifold''
or ``complex Lie group''
in this article,
we always mean a manifold or Lie group
modelled on a complex locally convex space,
which may be infinite-dimensional.
We shall say explicitly when
a manifold or Lie group
is assumed of finite dimension.
\begin{numba}\label{proj-tensor}
If $E$ is locally convex space and $F$ a finite-dimensional vector space
with basis $b_1,\ldots, b_\ell$,
give $F\otimes_\C E$ the locally convex vector topology making\vspace{-2mm}
\[
E^\ell\to F\otimes_\C E,\quad (v_1,\ldots, v_\ell)\mto \sum_{j=1}^\ell
b_j\otimes v_j\vspace{-2mm}
\]
an isomorphism of topological vector spaces. It is independent
of the choice of basis.
Thus: If $F=F_1\oplus F_2$, then\vspace{-1,5mm}
\[
F\otimes_\C E=(F_1\otimes_\C E)\oplus (F_2\otimes_\C E)\vspace{-1.5mm}
\]
as a topological vector space. Notably, $F_1\otimes_\C E$ is closed
in $F\otimes_\C E$.
\end{numba}
The following properties can be shown as in \cite{GaF}.
\begin{la}\label{basicR}
For every $R$-cia $\cA\sub C(\Sph,\C)$,
the following holds:
\begin{itemize}
\item[\rm(a)]
$\Sph\to\C$, $z\mto z$ has spectrum $\Sph$;
\item[\rm(b)]
$\C[z,z^{-1}]$ is dense in $\cA$;
\item[\rm(c)]
The set $\cM$ of algebra homomorphisms
$\cA\to \C$ is homeomorphic to $\Sph$;
\item[\rm(d)]
Identifying $\cM$ with $\Sph$,
the Gelfand transform
is the inclusion map $\iota\colon \cA\to C(\Sph,\C)$. $\,\square$
\end{itemize}
\end{la}
See \cite{Nee,GaN} for the following concept
(cf.\ also \cite{Rob,JFA}).
\begin{defn}
A complex Lie group $G$ modelled on a complex locally convex space
is called a \emph{BCH-Lie group}
if it has an exponential function which is a local diffeomorphism of
complex analytic manifolds at~$0$.
\end{defn}
Then $\exp_G^{-1}(\exp_G(x)\exp_G(y))=x+y+\frac{1}{2}[x,y]+\cdots$
is given by the BCH-series for small $x,y\in\cg$ (see \cite{GaN}).\\[2mm]
The following fact from \cite{CIA}
is one source of BCH-Lie groups.
\begin{numba}
For every complex continuous inverse algebra
$\cA$ which is Mackey complete,
the unit group $\cA^\times$ is a BCH-Lie group.
\end{numba}
For example, $\cA^{n\times n}$ is a cia for each
cia $\cA$, and thus $\GL_n(\cA)=(\cA^{n\times n})^\times$
is a BCH-Lie group if $\cA$ is Mackey complete.
\begin{defn}
A locally convex, complex topological Lie algebra $\cg$
is called BCH if the BCH-series converges
to a $\C$-analytic function $U\times U\to\cg$
for some open $0$-neighbourhood $U\sub \cg$ (cf.\ \cite{Rob}).
\end{defn}
See \cite{GaN} for the following fact (cf.\ also \cite{Rob}).
\begin{la}(Integration of Lie subalgebras)
Let $\cg$ be a BCH-Lie algebra.
If there exists an injective, continuous
Lie algebra homomorphism $\alpha\colon \cg\to L(H)$
to the Lie algebra $L(H)$ of some complex
BCH-Lie group~$H$,
then $G:=\langle\exp_H(\alpha(\cg))\rangle\leq H$
can be made a BCH-Lie group with $L(G)\cong \cg$.
\end{la}
\begin{numba}\label{details-BCH}
By Ado's Theorem, every finite-dimensional $\C$-Lie algebra $\cg$
is isomorphic to a Lie subalgebra of $\C^{n\times n}$ for
some $n\in\N$. Now $\cg\sub\C^{n\times n}$.\\[3mm]
For each Mackey-complete $R$-cia $\cA$, the Lie subalgebra
\[
\cg_\cA=\cg\otimes_\C\cA
\]
of $\C^{n\times n}\otimes_\C \cA$
is a closed vector subspace. As $\cA^{n\times n}\cong T_{\bf 1}(\GL_n(\cA))$
is BCH, also the BCH-series
of $\cg_\cA$ converges on an open $0$-neighbourhood
and thus $\cg_\cA$ is BCH.
Let $G$ be a complex Lie group with Lie algebra $T_e(G)\cong \cg$.
The inclusion map
\[
\cg_\cA=\cg\otimes_\C \cA\to\cg\otimes_\C C(\Sph,\C)=C(\Sph,\cg)
\]
is an
injective continuous Lie algebra homomorphism,
where $C(\Sph,\cg)\cong L(C(\Sph,G))$ for the BCH-Lie group
$C(\Sph ,G)$.
Hence:
\[
G_\cA:=\langle \exp_G\circ \hspace*{.3mm}f\colon f\in \cg_\cA\rangle\leq C(\Sph,G)
\]
can be made a BCH-Lie group with Lie algebra $L(G_\cA)\cong \cg_\cA$.\\[2mm]
If $\cA$ is decomposing, we have subgroups
\[
G_\cA^+:=\langle \exp_G\circ \hspace*{.3mm}f\colon f\in \cg\otimes_\C \cA^+ \rangle,\;\;
G_\cA^\ominus:=\langle \exp_G\circ \hspace*{.3mm}f\colon f\in \cg\otimes_\C \cA^\ominus\rangle
\]
of $G_\cA$ which are Lie groups with Lie algebras
$\cg^+:=\cg\otimes_\C \cA^+$ and $\cg^\ominus:=\cg\otimes_\C\cA^\ominus$,
respectively.
\end{numba}
%
%
%
%
%
%
%
%
%
\section{Uniqueness of decompositions}
The following
well-known fact will be used repeatedly:
\begin{numba}\label{rera}
\emph{Let $r<1<R$ and $f\colon U\to\C$
be a continuous function on the open annulus
$U:=\{z\in \C\colon r<|z|<R\}$
such that $f$ is holomorphic on $U\setminus\Sph$.
Then $f$ is holomorphic.}
\end{numba}
For the proof, one verifies the hypotheses
of Morera's Theorem for
$z\mto f(e^z)$
on $\{z\in\C\colon \ln(r)<\Repart(z)<\ln(R)\}$.
\begin{la}\label{uni1}
Let $G$ be a complex Lie group.
Let $f_+,g_+\colon \wb{\bD}\to G$
be continuous functions holomorphic
on $\bD$ and $f_-,g_-\colon \wh{\C}\setminus \bD\to G$
be continuous functions holomorphic on $\wh{\C}\setminus \wb{\bD}$,
such that
\[
f_+(z)f_-(z)=g_+(z)g_-(z)\quad\mbox{for all $\,z\in \Sph\,$ and}\quad
f_-(\infty)=g_-(\infty).
\]
If holomorphic functions $G\to\C$ separate points
on~$G$, then $f_+=g_+$ and $f_-=g_-$.
\end{la}
\begin{proof}
The function
\[
h\colon \wh{\C}\to G,\quad z\mto\left\{
\begin{array}{cl}
g_+(z)^{-1}f_+(z) & \mbox{if $\, z\in \wb{\bD}$;}\\
g_-(z)f_-(z)^{-1} & \mbox{if $\, z\in\wh{\C}\setminus \bD$}
\end{array}\right.
\]
is continuous. Now $\phi\circ h\colon \wh{\C}\to \C$
is continuous for each holomorphic function $\phi\colon G\to\C$
and holomorphic on $\wh{\C}\setminus \Sph$,
whence $\phi\circ h$ is holomorphic (see Lemma~\ref{uni1})
and hence constant as~$\wb{\C}$ is compact.
Thus $h$ is constant, with value $g_-(\infty)f_-(\infty)^{-1}=e$.
\end{proof}
Passing to transposes,
Theorems~1.1 and 1.2
in \cite[Chapter~VIII]{GaF},
applied to continuous functions,
show the following:
\begin{la}\label{uni2}
Let $A_+,A_-,\wt{A}_+,\wt{A}_-\in C(\Sph,\GL_n(\C))$
and $\kappa_1\geq \kappa_2\geq \cdots\geq \kappa_n$
as well as $\nu_1\geq \nu_2\geq\cdots\geq \nu_n$
be integers.
Let $D(z)$ be the diagonal matrix
$\diag(z^{\kappa_1},\ldots,z^{\kappa_n})$
and $\wt{D}(z):=\diag(z^{\nu_1},\ldots, z^{\nu_n})$
for $z\in\C$.
Assume that $A_+$ and $\wt{A}_+$
admit continuous extensions $\wb{\bD}\to\GL_n(\C)$
$($denoted by the same symbol$)$
which are holomorphic on $\bD$.
Assume that $A_-$ and $\wt{A}_-$
admit continuous extensions $\wh{\C}\setminus \bD\to\GL_n(\C)$
$($denoted by the same symbol$)$
which are holomorphic on $\wh{\C}\setminus \wb{\bD}$.
If
\[
A_+(z)D(z)A_-(z)=\wt{A}_+\wt{D}(z)\wt{A}_-\quad\mbox{for all $\,z\in\Sph$,}
\]
then $\kappa_j=\nu_j$ for all $j\in\{1,\ldots,n\}$.
Moreover,
there exists a continuous function
$C\colon \C\to \GL_n(\C)$, $C(z)=(C_{jk})_{j,k=1}^n$
for $z\in\C$ such that
\begin{eqnarray}
\wt{A}_+(z) & = & A_+(z)C(z)\qquad\qquad\quad\hspace*{8.45mm}
\mbox{for all $z\in \wb{\bD}$;}\label{ba1}\\
\wt{A}_-(z) &=& D(z)^{-1}C(z)^{-1}D(z)A_-(z)\quad\mbox{for all
$z\in \C\setminus \bD$}\label{ba2}
\end{eqnarray}
and the following conditions are satisfied
for all $j,k\in\{1,\ldots,n\}$:
\begin{itemize}
\item[\rm(a)]
$c_{kj}=0$ if $\kappa_k<\kappa_j$;
\item[\rm(b)]
$c_{kj}$ is constant if $\kappa_k=\kappa_j$;
\item[\rm(c)]
$c_{kj}$ is a polynomial
of degree $\leq \kappa_k-\kappa_j$, if $\kappa_k>\kappa_j$.
\end{itemize}
\end{la}
Note that if the matrix elements of
$C$ satisfy the conditions~(a)--(c)
for all $j,k\in\{1,\ldots,n\}$,
then also the matrix elements of
$z\mto C(z)^{-1}$.
\section{Birkhoff decompositions in {\boldmath$\GL_n(\cA)$}}
Let $\cA^-:=\C 1 +\cA^\ominus$.\\[2.3mm]
\begin{proof} (of Proposition~\ref{birkho-1}).
Holomorphic functional calculus
provides an exponential function $\exp$
and a logarithm function for~$\cA$,
and one finds that~$\exp$ is a local $C^\infty_\C$-diffeomorphism
at~$0$ (see \cite{CIA}).
The Lie group $\cA^\times$ being abelian,
$\exp\colon \cA\to\cA^\times$
is a group homomorphism.
Hence
$\exp(\cA)$ is an open
subgroup of $\cA^\times$. If $g\in \exp(\cA)$,
then $g=\exp(x)$ with $x\in\cA=\cA^+\oplus\cA^\ominus$.
Write $x=x_++x_\ominus$. Then $g=\exp(x_+)\exp(x_\ominus)$
with $\exp(x_+)\in (\cA^+)^\times$ and $\exp(x_\ominus)\in(\cA^-)^\times$.
Since $\C[z,z^{-1}]\cap \cA^\times$ is dense,
it remains to factor Laurent polynomials $g$
in this dense set:
$g=z^\kappa p$ with $p(0)\not=0$ then
\[
\qquad\quad \qquad g(z)=\underbrace{c \prod_{|w|>1}(z-w)^{\nu_w}}_{=:g_+(z)}
\,z^\kappa \, \underbrace{\prod_{|w|<1}(z-w)^{\nu_w}}_{=:g_-(z)}.
\]
\end{proof}
\begin{proof} (of Proposition~\ref{birkho-n}).
See Theorem~2.1 in \cite[Chapter~VIII]{GaF}
for the first assertion.
For the final conclusion,
note that the proof of the ``special
case'' of the cited theorem (on pages 189--190)
applies without changes.
In the third line of \cite[page 191]{GaF},
we replace the Lemma~5.1 of loc.\,cit.\
with our hypothesis that $\GL_n(\cA^+)\GL_n(\cA^-)$
is an identity neighbourhood in $\GL_n(\cA)$.
\end{proof}
\begin{la}\label{go-down}
Let $\cA$ and $\cB$ be decomposing $R$-cias
such that $\cA\sub\cB$.
Given $A\in \GL_n(\cA)$,
assume that
\[
A(z)=A_+(z)D(z)A_-(z)=\wt{A}_+(z)\wt{D}(z)\wt{A}_-(z),
\]
where $A_+\in\GL_n(\cA^+)$, $A_-\in\GL_n(\cA^-)$,
$\wt{A}_+\in \GL_n(\cB^+)$, $\wt{A}_-\in \GL_n(\cB^-)$
and $D(z)=\diag(z^{\kappa_1},\ldots,z^{\kappa_n})$,
$\wt{D}(z)=\diag(z^{\nu_1},\ldots,z^{\nu_n})$
with integers $\kappa_1\geq \cdots\geq \kappa_n$
and $\nu_1\geq\cdots\geq \nu_n$.
Then $D=\wt{D}$,
$\wt{A}_+\in \GL_n(\cA^+)$, and $\wt{A}_-\in \GL_n(\cA^-)$.
\end{la}
\begin{proof}
By Lemma~\ref{uni2}, we have $D=\wt{D}$.
Let $C$ be as in Lemma~\ref{uni2}.
Since $C$ and $z\mto C(z)^{-1}$ are polynomial,
we have $C|_\Sph\in \cA^{n\times n}$
and $C^{-1}|_\Sph\in\cA^{n\times n}$,
entailing that $C\in\GL_n(\cA)$.
We now deduce from~(\ref{ba1})
and (\ref{ba2}) that
$\wt{A}_+\in \GL_n(\cA)\cap \GL_n(\cB^+)=\GL_n(\cA^+)$
and $\wt{A}_-\in\GL_n(\cA)\cap\GL_n(\cB^-)=\GL_n(\cA^-)$.
\end{proof}
\begin{prop}\label{birkho-lim}
Let $(\cA_j)_{j\in\N}$
be a sequence of Mackey complete,
decomposing $R$-cias.
Let $n\in\N$ such that $\GL_n(\cA_j^+)\GL_n(\cA^-_j)$
is an identity neighbourhood in $\GL_n(\cA_j)$
for each $j\in\N$.
If $\cA_1\supseteq \cA_2\supseteq\cdots$
and all inclusion maps are continuous,
then also the decomposing $R$-cia $\cA:=\bigcap_{j\in\N}\cA_j=\pl\,\cA_j$
has the property that
$\GL_n(\cA^+)\GL_n(\cA^-)$ is an identity
neighbourhood in $\GL_n(\cA)$.
\end{prop}
\begin{proof}
Let $A\in \GL_n(\cA)$.
Then $A\in\GL_n(\cA_j)$ for
each $j\in\N$.
By Proposition~\ref{birkho-n},
we have
\[
A=A_{j,+}D_jA_{j,-}
\]
with $D_j$ (as described in the proposition)
taking values in diagonal matrices)
and $A_{j,+}\in\GL_n(\cA_j^+)$, $A_{j,-}\in\GL_n(\cA_j^-)$.
By Lemma~\ref{go-down},
we have
\[
A_+:=A_{1,+}\in \GL_n(\cA_j^+)\quad \mbox{and}\quad
A_-:=A_{1,-}\in\GL_n(\cA_j^-)
\]
for each $j\in \N$, whence
\[
A_+\in\bigcap_{j\in \N}\GL_n(\cA_j^+)=\GL_n(\cA^+)\quad\mbox{and}\quad
A_-\in\bigcap_{j\in \N}\GL_n(\cA_j^-)=\GL_n(\cA^-).
\]
The inclusion map $\cA\to\cA_1$
is a continuous homomorphism of unital
algebras, whence the inclusion map
$j\colon \GL_n(\cA)\to\GL_n(\cA_1)$
is a continuous group homomorphism.
Hence
\[
\Omega:=j^{-1}(\GL_n(\cA_1^+)\GL_n(\cA^-))
\]
is an identity neighbourhood in $\GL_n(\cA)$.
By the preceding discussion,
$\Omega\sub \GL_n(\cA^+)\GL_n(\cA^-)$.
\end{proof}
We recall that parts (a) and (b)
of the following proposition can also be found in~\cite{PaS}.
\begin{prop}
Let $n\in\N$.
Consider the following situations:
\begin{itemize}
\item[\rm(a)]
$G:=C^\infty(\Sph,\GL_n(\C))\cong \GL_n(C^\infty(\Sph,\C))$,
$G^+:=C^\infty(\Sph,\GL_n(\C))^+$,
$G^\ominus:=C^\infty(\Sph,\GL_n(\C))^\ominus$,
and $G^-:=C^\infty(\Sph,\GL_n(\C))^-$; or:
\item[\rm(b)]
$G:=C^\omega(\Sph,\GL_n(\C))\cong \GL_n(C^\omega(\Sph,\C))$,
$G^+:=C^\omega(\Sph,\GL_n(\C))^+$,
$G^\ominus:=C^\omega(\Sph,\GL_n(\C))^\ominus$,
and $G^-:=C^\omega(\Sph,\GL_n(\C))^-$; or:
\item[\rm(c)]
$G:=\cO(\C^\times,\GL_n(\C)))$,
$G^+:=\cO(\C^\times,\GL_n(\C))^+$,
$G^\ominus:=\cO(\C^\times,\GL_n(\C))^\ominus$,
and $G^-:=\cO(\C^\times,\GL_n(\C))^-$.
\end{itemize}
Then $G^+G^\ominus$ is open in~$G$.
Each $g\in G$ can be written as
\[
g=g_+Dg_-
\]
with $g_+\in G^+$, $g_-\in G^-$
and $D\colon \Sph\to\GL_n(\C)$,
$z\mto\diag(z^{\kappa_1},\ldots, z^{\kappa_n})$
with integers $\kappa_1\geq \kappa_2\geq \cdots$.
The integers $\kappa_1,\ldots,\kappa_n$
are uniquely determined.
\end{prop}
\begin{proof}
(a) The hypotheses of Proposition~\ref{birkho-lim}
are satisfied by $C^\infty(\Sph,\C)=\bigcap_{m\in\N}\cW(m)$,
whence $G^+G^-=G^+G^\ominus$
is an identity neighbourhood
in~$G$ and hence
open, being a $G^+\times G^\ominus$
orbit if we let $G^+$ act on~$G$ by multiplication
on the left and $g\in G^\ominus$ multiplication with $g^{-1}$
on the right.
The remaining assertions now follow from Proposition~\ref{birkho-n}.\\[2.3mm]
(b) The openness of $G^+G^\ominus$
is a special case of Theorem~~\ref{main-1}
(and we shall not use the current statement
before the theorem is proved).
Since $\cA$, being a Silva space,
is complete, the hypotheses of Proposition~\ref{birkho-n}
are satisfied.\\[2.3mm]
(c) The proof is very similar to the beginning
of the proof of Theorem~\ref{main-1}.
Details will be given in a later version
of the preprint.
\end{proof}
\section{Decompositions of loops in {\boldmath$G$}}\label{sec-G}
\begin{proof} (of Proposition~\ref{deco-easy}).
We know from Lemma~\ref{uni1} that
$\pi$ is injective.\\[1.9mm]
\emph{$\pi$ is a local complex-analytic diffeomorphism} (and hence a complex-analytic diffeomorphism)
if it is so at $(e,e)$. In fact, $\pi$ is equivariant for the $\C$-analytic
$G^+_\cA\times G^\ominus_\cA$-actions
\[
(g_+,g_\ominus).(x,y):=(g_+x, yg_\ominus^{-1})
\quad\mbox{and}\quad
(g_+,g_\ominus).z:=g_+zg_\ominus^{-1}
\]
on $G^+_\cA\times G^\ominus_\cA$ and $G_\cA$,
respectively; the first of these is a transitive
action. Hence $\pi$ will be a local complex-analytic
diffeomorphism at each point in its domain
if it is a local complex-analytic diffeomorphism at $(e,e)$.
We have
\[
T_{(e,e)}(G_\cA^+\times G_\cA^\ominus)\cong
T_eG_\cA^+\times T_eG_\cA^\ominus
=\cg_\cA^+\times \cg_\cA^\ominus.
\]
using this identification,
the tangent map $T_{(e,e)}\pi$ is the map
\[
\cg_\cA^+\times \times \cg_\cA^\ominus\to\cg_\cA,\quad (f,g)\mto f+g,
\]
which is an isomorphism of topological
vector spaces.
We can now apply the inverse function theorem for complex-analytic
mappings between open subsets of complex Banach spaces,
to conclude that~$\pi$ is a local complex-analytic
diffeomorphism at $(e,e)$.
\end{proof}
\begin{proof} (of Theorem~\ref{main-1}).
The restriction map
\[
\rho\colon \cO(\C^\times,G)_0\to C^\infty(\Sph,G),\quad f\mto f|_\Sph
\]
is a homomorphism of groups.
It is $C^\infty$ by a suitable
exponential law (as in \cite{AGS}),
since the corresponding map $\rho^\wedge\colon
\cO(\C^\times, G)_0\times\Sph\to G$
is smooth.
The differential $T_e\rho_1$ corresponds to the
restriction map $\cO(\C^\times,\cg)\to C^\infty(\Sph,\cg)$,
whence it is complex linear.
As a consequence, the smooth group homomorphism~$\rho$
is complex analytic. In particular,
$\rho$ is continuous,
whence it takes $\cO(\C^\times, G)_0$
into $C^\infty(\Sph,G)_0$.
The exponential map of
$C^\infty(\Sph,G)_0$
is
\[
\exp_\infty\colon C^\infty(\Sph,\cg)\to C^\infty(\Sph,G)_0,\quad
f\mto \exp_G\circ f.
\]
The exponential map of
$G_\cW$
is
\[
\exp_\cW\colon \cg_\cW\to G_\cW,\quad
f\mto \exp_G\circ f.
\]
The inclusion map $\iota \colon C^\infty(\Sph,\cg)\to\cg_\cW$
is continuous and linear, hence complex analytic.
We have
\[
(\exp_\cW\circ \iota)(f)=\exp_G\circ f=\exp_\infty(f)
\]
for each $f\in C^\infty(\Sph,\cg)$,
whence $\exp_\infty(f)\in G_\cW$.
Inside $C(\Sph,G)$, we therefore have
\[
C^\infty(\Sph,G)_0=\langle \exp_\infty(f)\colon f\in C^\infty(\Sph,\cg)\rangle
\sub \langle \exp_G\circ f\colon f\in\cg_\cW\rangle=G_\cW.
\]
We can therefore consider the inclusion map
$\lambda\colon C^\infty(\Sph,G)_0\to G_\cW$,
which is homomorphism of groups.
It is complex analytic on some open
identity neighbourhood (and hence
complex analytic),
as we have a commuting diagramme
\[
\begin{array}{rcl}
C^\infty(\Sph,G)_0 & \stackrel{\lambda}{\rightarrow} & G_\cW\\
\exp_\infty\uparrow & & \uparrow \exp_\cW\\
C^\infty(\Sph,\cg) & \stackrel{\iota}{\rightarrow} & \cg_\cW
\end{array}
\]
in which $\exp_\cW$ and $\iota$
are complex analytic and $\exp_\infty$
is a local complex analytic diffeomorphism at~$0$,
with $\exp_\infty(0)=e$.
By the preceding, the restriction map
\[
\lambda\circ \rho\colon \cO(\C^\times,G)\to G_\cW,\quad f\mto f|_\Sph
\]
is a complex analytic homomorphism of groups.
In particular, $\lambda\circ \rho$ is continuous,
whence $\Omega:=(\lambda\circ \rho)^{-1}(G_\cW^+G_\cW^\ominus)$
is an open identity neighbourhood in $\cO(\C^\times,G)$.\\[2.3mm]
To see that \emph{the image $\im(p)$ is open},
it suffices to show that $\im(p)$
equals the set~$\Omega$ just defined.
Thus, we need to prove
that $\im(p)$ contains all $f\in \cO(\C^\times,G)$
such that $g:=f|_{\Sph}\in G_\cW^+ G_\cW^\ominus$.
To achieve this, we write $g=g_+g_\ominus$
with $g_+\in G_\cW^+$ and $g_\ominus\in G_\cW^\ominus$.
Then
\[
\wh{f}_+(z):=\left\{
\begin{array}{cl}
g_+^<(z) &\mbox{\,if $\,z\in \wb{\bD}$,}\\
f(z)g_\ominus^>(z) & \mbox{\,if $\,z\in \C\setminus \bD$}
\end{array}\right.
\]
defines a continous
function $\C^\to G$
which is complex analytic on $\C \setminus\Sph$
and hene complex analytic on all of~$\C$.
Likewise,
\[
\wh{f}_\ominus(z):=\left\{
\begin{array}{cl}
g_+^<(z)^{-1}f(z) &\mbox{\,if $\,z\in \wb{\bD}$,}\\
g_\ominus^>(z) & \mbox{\,if $\,z\in \wh{\C}\setminus \bD$}
\end{array}\right.
\]
defines a continous
function $\wh{\C}\setminus\{0\} \to G$
which is complex analytic on $(\wh{\C}\setminus\{0\})\setminus\Sph$
and hence complex analytic on all of~$\wh{\C}\setminus\{0\}$.
Moreover, $\wt{f}_\ominus(\infty)=g_\ominus^>(\infty)=e$.
Then
\[
f=f_+f_\ominus
\]
by the Identity Theorem,
as both functions are complex analytic
and coincide on $\Sph$.
Thus $f=p(\wh{f}_+,\wh{f}_\ominus)\in\im(p)$.\\[2.3mm]
By Lemma~\ref{uni1}, \emph{$p$ is injective}.\\[2.3mm]
\emph{Analyticity of $p^{-1}$, say of $f\mto \widetilde{f}_+$}:
By the exponential law, we only need to show that the corresponding
map
\begin{equation}\label{onlyeq}
\im(p)\times \C\to G,\quad (f,z)\mto \wt{f_+}(z)
\end{equation}
of two variables is
complex analytic.
Actually, complex analyticity of the restriction
of the map in~(\ref{onlyeq})
to $\im(f)\times \bD$,
as map
is equivariant for natural
$\C^\times$-action, $(w.f)(z):=f(wz)$.
We have a $\C$-analytic composition
\[
\begin{array}{cccccccc}
\psi \colon & \im(p)& \longrightarrow & G_\cW& \longrightarrow & G_\cW^+ &
\longrightarrow & \cO(\bD,G)\\
&& f\mto f|_\Sph & & f\mto f_+ & & f\mto \wt{f}|_\bD &
\end{array}
\]
of $\C$-analytic maps.
The leftmost map used here
is complex analytic as a restriction of $\lambda\circ\rho$.
The second map is complex analytic by Proposition~\ref{deco-easy}.
The final map is a group homomorphism
which is complex analytic on some
identity neighbourhood (and hence
complex analytic)
because it is the upper map in the commutative diagramme
\[
\begin{array}{rcl}
G_\cW^+ & \rightarrow & \cO(\bD,G)\\
\exp_{\cW^+}\uparrow & & \uparrow \EXP\\
\cg_\cW^+ & \rightarrow & \cO(\bD,\cg)
\end{array}
\]
where $\exp_{\cW^+}\colon \cg_\cW^+\to G_\cW^+$,
$f\mto \exp_g\circ f$ is a local
complex analytic diffeomorphism at~$0$,
the exponential map
$\EXP\colon \cO(\bD,\cg)\to\cO(\bD,G)$, $f\mto \exp_G\circ f$
is complex analytic and the bottom map is the composition
\[
\cg_\cW^+\to \cO_b(\bD,\cg)\to \cO(\bD,\cg)
\]
of the continuous complex linear maps
$\cg_\cW^+\to \cO(\bD,\cg)$, $f\mto f^<$
(noting that $\|f^<\|_\infty\leq \|f\|_\cW$
as a consequence~(\ref{simply-plus}))
and the complex linear inclusion map
$\cO_b(\bD,\cg)\to \cO(\bD,\cg)$
which is continuous as the topology of
compact convergence on $\cO_b(\bD,\cg)$
is coarser than the topology of uniform convergence.
%
It remains to observe that (\ref{onlyeq})
is the map
\[
\psi^\wedge\colon \im(p)\times \bD\to G,\quad (f,z)\mto \psi(f)(z)
\]
which is $\C$-analytic by the Exponential Law for $\cO(\bD,G)$.\\[2.3mm]
Complex analyticity of the second component of
$p^{-1}$ can be shown
analogously, exploiting
that the map $\phi \wb{\C}\setminus \{0\}\to\C$, $z\mto\frac{1}{z}$
(with $1/\infty:=0$)
is an isomorphism of complex analytic manifolds
and the map
$\cO(\C,G)\to \cO(\wh{\C}\setminus \{0\},G)$,
$f\mto f\circ \phi$ an isomorphism
of complex analytic Lie groups.
\end{proof}
%
%
%
%
\noindent
We shall use a result concerning
complex analyticity of mappings
on open subsets of LB-spaces
(taken from \cite{Dah}).
\begin{la}[Dahmen's Theorem]
Let $E_1\sub E_2\sub\cdots$ and $F_1\sub F_2\sub\cdots$ be complex Banach
spaces such that all inclusion maps have operator norm $\leq 1$.
Assume that the locally convex direct limit topology on $\bigcup_{n\in\N}E_n$
and $F:=\bigcup_{n\in\N}F_n$ is Hausdorff. Let $r>0$. Then
a map $f\colon \bigcup_{n\in \N}B^{E_n}_r(0)\to F$
is complex analytic if
$f|_{B^{E_n}_r(0)} \colon B^{E_n}_r(0)\to F$
is complex analytic \emph{and bounded}
for each $n\in \N$. $\square$
\end{la}
We recall a well-known concept.
\begin{defn}
If $f\colon X\to Y$
is a mapping between
metric spaces $(X,d_X)$
and $(Y,d_Y)$, let
\[
\Lip(f):=\sup\left\{
\frac{d_Y(f(x),f(y))}{d_X(x,y)}\colon \mbox{for $x,y\in X$ with $x\not=y$}\right\}
\in [0,\infty].
\]
\end{defn}
Then $f$ is Lipschitz continuous if and only if
$\Lip(f)<\infty$
and $\Lip(f)$ is the smallest Lipschitz constant for~$f$
in the case.
Also the following fact is well known (see, e.g.,
\cite{GaN}):
\begin{numba}\label{difflip}
Let $(E,\|\cdot\|_E)$
and $(F,\|\cdot\|_F)$
be Banach spaces and $U\sub E$
be a convex open subset.
If $f\colon U\to F$ is
continuously Fr\'{e}chet differentiable
and $f'(x)\in\cL(E,F)$ its derivative at $x\in U$,
then
\[
\Lip(f)=\sup\{\|f'(x)\|_{\op}\colon x\in U\}.
\]
\end{numba}
We shall always assume that the
norm $\|\cdot\|_\cg$ on a Banach--Lie algebra~$\cg$
is chosen submultiplicative,
i.e., that
\[
\|[x,y]\|_\cg\leq \|x\|_\cg\|y\|_\cg\quad\mbox{for all $\,x,y\in\cg$.}
\]
\begin{example}\label{Wiener-sm}
For each complex Banach--Lie algebra $\cg$
with submultiplicative norm $\|\cdot\|_\cg$,
the norm $\|\cdot\|_\cW$
on $\cg_\cW$ is submultiplicative.
In fact, the map
\[
\cg_\cW\to \ell^1(\Z,\cg),\quad f\mto (\hat{f}_k)_{k\in \Z}
\]
is isometric and an isomorphism of complex
Lie algebras, when the norm
$\|\cdot\|_{\ell^1,\cg}$ defined via
\[
\|(a_k)_{k\in\Z}\|_{\ell^1,\cg}:=\sum_{k\in\Z}\|a_k\|_\cg
\]
is used on $\ell^1(\Z,\cg)$.
The Lie bracket on $\ell^1(\Z,\cg)$ is
the Cauchy--Lie bracket, $[(a_k)_{k\in\Z},(b_k)_{k\in\Z}]:=(c_k)_{k\in \Z}$
with
\[
c_k:=\sum_{\ell\in\Z}[a_\ell, b_{k-\ell}].
\]
Setting $\alpha_k:=\|a_k\|_\cg$
and $\beta_k:=\|b_k\|_\cg$ for $k\in\Z$,
we have $\alpha:=(\alpha_k)_{k\in\Z}\in \ell^1$
and $\beta:=(\beta_k)_{k\in\Z}\in \ell^1$.
Now
\begin{eqnarray*}
\|(c_k)_{k\in\Z}\|_{\ell^1,\cg}&\leq& \sum_{\ell\in\Z}\|[a_\ell,b_{k-\ell}]\|_\cg
\leq \sum_{\ell\in\Z}\|a_\ell\|_\cg\|b_{k-\ell}\|_\cg\\
&=& \|\alpha * \beta\|_{\ell^1}\leq \|\alpha\|_{\ell^1}\|\beta\|_{\ell^1}
=\|(a_k)_{k\in \Z}\|_{\ell^1,\cg}\|(b_k)_{k\in\Z}\|_{\ell^1,\cg},
\end{eqnarray*}
showing that $\|\cdot\|_{\ell^1,\cg}$
(and hence also $\|\cdot\|_{\cW}$ on $\cg_\cW$)
is submultiplicative.
If $f\in \cg_\cW$,
then $f=f_++f_\ominus$
with $f_+\in \cg_\cW^+$ and $f_\ominus\in \cg_\cW^\ominus$
given by
\[
f_+(z):=\sum_{k=0}^\infty \hat{f}_kz^k\quad\mbox{and}\quad
f_\ominus(z):=\sum_{k=1}^\infty\hat{f}_{-k}z^{-k}
\]
for $z\in\Sph$. Since $f_+^<\colon \wb{\bD}\to\cg$
and $f_-^>\colon \wh{\C}\setminus \bD\to\cg$
are given by (\ref{simply-plus})
and (\ref{simply-minus}), respectively
(with $f_-^>(\infty):=0$),
we see that
\begin{equation}\label{pr1wiener}
\|f_+^<\|_\infty\leq \|(\|\hat{f}_k\|_\cg)_{k\in\N}\|_{\ell^1}
\leq \|(\|\hat{f}_k\|_\cg)_{k\in \Z}\|_{\ell^1}=\|f\|_\cW
\end{equation}
and
\begin{equation}\label{pr2wiener}
\|f_\ominus^>\|_\infty\leq \|(\|\hat{f}_{-k}\|_\cg)_{k\in\N}\|_{\ell^1}
\leq \|(\|\hat{f}_k\|_\cg)_{k\in\Z}\|_{\ell^1}=\|f\|_\cW.
\end{equation}
\end{example}
\begin{la}\label{universal-BCH}
There exists $r>0$
with the following properties:
\begin{itemize}
\item[\rm(a)]
For each complex Banach--Lie algebra
$\cg$ with submultiplicative norm $\|\cdot\|_\cg$,
the BCH-series converges
on $B^\cg_r(0)\times B^\cg_r(0)$
to a complex-analytic mapping
\[
\mu_\cg\colon B^\cg_r(0)\times B^\cg_r(0)\to \cg,\quad
(x,y)\mto x*y.
\]
\item[\rm(b)]
If we write $x*y=x+y+R_\cg(x,y)$
for $(x,y)\in B^\cg_r(0)\times B^\cg_r(0)$
in the situation of~{\rm(a)},
then
\[
\Lip(R_\cg)\leq \frac{1}{4}.
\]
\end{itemize}
We are using the maximum
norm $\cg\times \cg\to[0,\infty[$, $(x,y)\mto \max\{\|x\|_\cg,\|y\|_\cg\}$
here for the calculation
of the Lipschitz constant.
\end{la}
\begin{proof}
Let $\alpha_\cg\colon \cg\times\cg\to \cg$
be the continuous linear map $(x,y)\mto x+y$.
By \ref{difflip},
the final assertion means that $\|\mu'_\cg(x,y)-\alpha_\cg\|_{\op}\leq\frac{1}{4}$
for all $(x,y)\in B^\cg_r(0)\times B^\cg_r(0)$.
The assertion can therefore be proved like
\cite[Lemma~3.5\,(a)]{Da2},
where the the constant
$\frac{1}{2}$ is used in place of~$\frac{1}{4}$.
\end{proof}
\begin{numba}
If $\cg$ is a Banach--Lie algebra
with submultiplicative norm
$\|\cdot\|_\cg$,
we endow $\cO_b(A_n,\cg)$
with the norm $\|\cdot\|_n$ given by
\[
\|f\|_n:=\max\{\|f\|_\infty,\|f|_\Sph\|_\cW\}.
\]
Using the compact-open $C^\infty$-topology on the function
spaces in the middle,
we get continuous complex linear maps
\[
(\cO_b(A_n,\cg),\|\cdot\|_\infty)
\to C^\infty(A_n,\cg)\to C^\infty(\Sph,\cg)\to \cg_\cW,
\]
where the first and final mappings are inclusion maps;
the map in the middle is the restriction map $f\mto f|_\Sph$.
Thus, the composition
\[
\rho\colon (\cO_b(A_n,\cg),\|\cdot\|_\infty)\to \cg_\cW,\quad f\mto f|_\Sph
\]
is a continuous complex linear map,
entailing that $\|\cdot\|_\cW\circ\rho$
is a continuous norm on $(\cO_b(A_n,\cg),\|\cdot\|_\infty)$.
As a consequence, the norm
\[
\|\cdot\|_n\colon \cB(A_n,\cg)\to [0,\infty[,\quad
f\mto\max\{\|f\|_\infty,\|f|_\Sph\|_\cW\}
\]
is equivalent to~$\|\cdot\|_\infty$.
Since $\|[f,g](x)\|_\cg=\|[f(x),g(x)]\|_\cg\leq \|f(x)\|_\cg\|g(x)\|_\cg\leq
\|f\|_\infty\|g\|_\infty$ and
\[
\|[f,g]|_\Sph\|_\cW\leq\|f|_\Sph\|_\cW\|g|_\Sph\|_\cW
\]
by Example~\ref{Wiener-sm},
the norm $\|\cdot\|_n$
is submultiplicative.
We now show that each $f\in \cO_b(A_n,\cg)$
admits a decomposition
\[
f=f_++ f_\ominus
\]
with $f_+,f_\ominus\in \cO_b(A_n,\cg)$
such that $f_+$ admits a complex analytic extension
$f_+^<\colon B^\C_{1+\frac{1}{n}}(0)\to\cg$
and $f_-$ admits a complex analytic extension
$f_-^>\colon \wh{\C}\setminus \wb{B}^\C_{1-\frac{1}{n}}(0)\to\cg$
with $f_-^>(\infty)=0$.
In fact, we write $g:=f|_\Sph\in\cg_\cW$
in the form $g=g_++ g_\ominus$
with $g_+\in\cg_\cW^+$ and $g_\ominus\in\cg_\cW^\ominus$.
Then
\[
f_+^<\colon B^\C_{1+\frac{1}{n}}(0)\to\cg,\quad
z\mto\left\{\begin{array}{cl}
g_+^<(z) & \mbox{\,if $\,z\in \wb{\bD}$;}\\
f(z)-g_\ominus^>(z) & \mbox{\,if $\,z\in B^\C_{1+\frac{1}{n}}(0)\setminus \bD$}
\end{array}\right.
\]
is a continuous function on $B^\C_{1+\frac{1}{n}}(0)$
which is holomorphic on $B^\C_{1+\frac{1}{n}}(0)\setminus\Sph$
and hence holomorphic.
Likewise,
\[
f_\ominus^>\colon \wh{\C}\setminus \wb{B}^\C_{1-\frac{1}{n}}(0)\to\cg,
\quad
z\mto\left\{\begin{array}{cl}
g_\ominus^>(z) & \mbox{\,if $\,z\in \wh{\C}\setminus \bD$;}\\
f(z)-g_+^<(z)& \mbox{\,if $\,z\in \bD$}
\end{array}\right.
\]
is a continuous function on $\wh{\C}\setminus \wb{B}^\C_{1-\frac{1}{n}}(0)$
which is holomorphic on $(\wh{\C}\setminus \wb{B}^\C_{1-\frac{1}{n}}(0))
\setminus\Sph$
and hence holomorphic.
We shall verify presently that $f_+:= (f_+^<)|_{A_n}$
and $f_\ominus := (f_-^>)|_{A_n}$
are in $\cO_b(A_n,\cg)$; in fact,
\begin{equation}\label{opno-proj}
\|f_+\|_n\leq 2\,\|f\|_n\quad \mbox{and}\quad \|f_-\|_n\leq
2\,\|f\|_n.
\end{equation}
By construction, $f=f_++ f_\ominus$.
We have $\|f_+|_\Sph\|_\cW=\|g_+\|_\cW\leq \|g\|_\cW\leq \|f\|_n\leq 2\,\|f\|_n$.
Since
\[
\|f_+(z)\|_\cg=\|g_+^<(z)\|_\cg\leq \|g\|_\cW\leq \|f\|_n\leq 2\,\|f\|_n
\]
for all $z\in \C$ with $1-\frac{1}{n}<|z|\leq 1$
and
\[
\|f_+(z)\|_\cg\leq \|f(z)\|_\cg+\|g_\ominus^>(z)\|_\cg\leq\|f\|_\infty+\|g\|_\cW\leq
\|f\|_\infty+\|f\|_n\leq 2\,\|f\|_n
\]
for all $z\in\C$ with $1\leq |z|<1+\frac{1}{n}$,
we see that $\|f_+\|_\infty\leq 2\,\|f\|_n$,
whence
$f_+\in\cO_b(A_n,\cg)$. Moreover,
$\|f_+\|_n=\max\{\|f_+|_\Sph\|_\cW,\|f_+\|_\infty\}\leq 2\,\|f\|_n$.
The proofs for $f_\ominus\in \cO_b(A_n,\cg)$
and the second inequality in (\ref{opno-proj})
are analogous.
\end{numba}
We shall also use a
well-known quantitative
version of the inverse function theorem:
\begin{la}\label{quanti-inv}
Let $(E,\|\cdot\|_E)$
be a Banach space, $x\in E$, $r>0$
and $f\colon B^E_r(x)\to E$
be a complex-analytic mapping such that
$f'(x)=\id_E$ and
\[
\Lip(f-\id_E)\leq \frac{1}{2}.
\]
Then $f(B^E_r(x))$ is open in~$E$,
the map $f\colon B^E_r(x))\to f(B^E_r(x))$
is a complex analytic diffeomorphism,
and
\begin{equation}\label{halfball}
B^E_{r/2}(f(x))\;\sub \; f(B^E_r(x))\; \sub \; B^E_{3r/2}(f(x)).
\end{equation}
\end{la}
\begin{proof}
By \cite[Lemma~6.1\,(a)]{IMP},
$f$ has open image,
is a homeomorphism onto its image,
and~(\ref{halfball}) holds.
Since $\|f'(y)-\id_E\|_{\op}\leq \Lip(f-\id_E)\leq\frac{1}{2}<1$,
we have $f'(y)\in\GL(E)$ for each $y\in B^E_r(x)$,
whence $f$ is a local complex-analytic diffeomorphism
by the classical inverse function theorem
for complex-analytic mappings and hence
a complex-analytic diffeomorphism.
\end{proof}
\begin{proof} (of Theorem~\ref{main-2}).
For $n\in\N$, let $\cg_n:=\cO_b(A_n,\cg)$,
endowed with the submultiplicative
norm $\|\cdot\|_n$ from above.
Let $r$, $\mu_{\cg_n}$
and $R_{\cg_n}$
be a in Lemma~\ref{universal-BCH}.
The continuous complex linear maps
\[
\pr^+_n\colon \cg_n\to \cg_n^+,\quad f\mto f^+
\]
and $\pr^\ominus_n\colon \cg_n\to\cg_n^\ominus$, $f\mto f^\ominus$
have operator norms $\|\pr^+_n\|_{\op},\|\pr^\ominus_n\|_{\op}\leq 2$,
by~(\ref{opno-proj}).
Hence
\[
R_n\colon B^{\cg_n}_{r/2}(0)\to \cg,\quad
x\mto R_{\cg_n}(\pr^+_n(x),\pr^\ominus_n(x))
\]
satisfies $\Lip(R_n)\leq 2\,\Lip(R_{\cg_n})\leq \frac{1}{2}$.
For all $x\in B^\cg_{r/2}(x)$,
we have $x_+=\pr^+_n(x)\in B^{\cg_n}_r(0)$
and $x_\ominus=\pr^\ominus_n(x)\in B^{\cg_n}_r(0)$.
Moreover,
\[
\mu_{\cg_n}(\pr^+_n(x),\pr^-_n(x))=x+R_n(x),
\]
since $x_+ *x_\ominus= x_++x_\ominus+R_{\cg_n}(x_+,x_\ominus)
=x+R_n(x)$. Hence, the complex-analytic map
\[
h_n\colon B^{\cg_n}_{r/2}(0)\to\cg_n,\quad x\mto \pr^+_n(x)*\pr^\ominus_n(x)
\]
is of the form
\[
h_n=\id_{\cg_n}+R_n\quad \mbox{with $\Lip(R_n)\leq \frac{1}{2}$
and $h_n'(0)=\id_{\cg_n}$.}
\]
by Lemma~\ref{quanti-inv},
$h_n$ is a complex-analytic diffeomorphism
onto its open image $h_n(B^{\cg_n}_{r/2}(0))$,
and the latter contains $B^{\cg_n}_{r/4}(0)$
as a subset.
Now $C^\omega(\Sph,\cg)=\dl\,\cg_n$.
Identifying holomorphic functions on annuli
with their restrictions to $\Sph$, we simply have $\cg_n\sub \cg_{n+1}$
for each $n\in\N$ and
$C^\omega(\Sph,\cg)=\bigcup_{n\in \N}\cg_n$.
Using these identifications,
\[
U:=\bigcup_{n \in\N}B^{\cg_n}_{r/2}(0)\quad\mbox{and}\quad
V:=\bigcup_{n\in \N}B^{\cg_n}_{r/4}(0)
\]
are open zero-neighbourhoods in $C^\omega(\Sph,\cg)$.
We have a well-defined injective map
\[
h \colon U\to C^\omega(\Sph,\cg),\quad f\mto h_n(f)\quad\mbox{if $\,f\in B^{\cg_n}_{r/2}(0)$}
\]
which is holomorphic (by Dahmen's Theorem)
since $h|_{B^{\cg_n}_{r/2}(0)}=h_n$
is holomorphic and bounded for each $n\in\N$.
By continuity of~$h$, the preimage $W:=h^{-1}(V)$
is open in~$U$.
Now $h|_W \colon W\to V$
is a continuous bijection.
The map
\[
g\colon V\to C^\omega(\Sph,\cg),\quad f\mto h_n^{-1}(f)\quad\mbox{if
$\,f\in B^{\cg_n}_{r/4}(0)$}
\]
is holomorphic (by Dahmen's Theorem),
since $g|_{B^{\cg_n}_{r/4}(0)}=h_n^{-1}|_{B^{\cg_n}_{r/4}(0)}$
is holomorphic and bounded for each $n\in\N$.
By definition, $g=(f|_W)^{-1}$.
Thus $f|_W\colon W\to V$
is a complex-analytic diffeomorphism.
Abbreviate $\cA:=C^\omega(\Sph,\C)$.
Let $\exp_+$ and $\exp_\ominus$
be the exponential
functions of the BCH-Lie groups $G_\cA^+$
and $G_\cA^-$, respectively.
We have a commutative diagramme
\[
\begin{array}{rcl}
G_\cA^+\times G_\cA^\ominus & \stackrel{\pi}{\longrightarrow} & C^\omega(\Sph,G)\\
\exp_+\times\exp_\ominus \uparrow\;\;\;\; & & \;\; \hspace*{-.2mm}\uparrow\\
\cg_\cA^+\times \cg_\cA^\ominus & & \;\;\hspace*{+.2mm}| \\
(\pr^+,\pr^\ominus) \uparrow\;\;\;\; & & \;\;\hspace*{+.2mm}|\\
W \;\;\; & \stackrel{h|_W}{\longrightarrow} & \;V,
\end{array}
\]
where the composition of the vertical maps on the left
is a local complex-analytic diffeomorphism at~$0$
and the vertical map on the right-hand side
is the exponential function of the Lie group $C^\omega(\Sph,G)$,
which is a complex-analytic function.

The assertion follows.
\end{proof}
\begin{proof} (of Theorem~\ref{main-3}).
Abbreviate $\cA:=C^\infty(\Sph,\C)$.
If $f\in G_\cW^+G_\cW^\ominus\cap C^\infty(\Sph,G)$,
we have
\[
f=g_+g_\ominus
\]
with certain $g_+\in G_\cW^+$ and $g_\ominus\in G_\cW^\ominus$.
On the other hand, since $G\sub\GL_n(\C)$,
we have
\[
f=f_+Df_-
\]
for suitable elements $f\in \GL_n(\cA^+)$, $f_-\in \GL_n(\cA^-)$,
and $D(z)=\diag(z^{\kappa_1},\ldots,z^{\kappa_n})$
with $\kappa_1\geq\cdots\geq \kappa_n$.
Since $g_+\in \GL_n(\cA^+)$ and $g_\ominus\in \GL_n(\cA^-)$,
Proposition~\ref{uni2}
shows that $\kappa_1=\cdots=\kappa_n=0$
and $f_+=g_+C$, $f_-=C^{-1}g_-$
for some constant matrix $C \in \GL_n(\C)$.
Thus $g_+\in \GL_n(\cA^+)\cap G_\cW^+$
and $g_-\in \GL_n(\cA^-)\cap G_\cW^\ominus$.
For $g\in G_\cA$ close to~$e$
in $G_\cW$
(which defines an identity neighbourhood
in $C^\infty(\Sph,G)$),
one can show that $g_+\in C^\infty(\Sph,G)^+$
and $g_-\in C^\infty(\Sph,G)^\ominus$.
The idea is to write $g_+=\exp_G\circ h$
with $h\in C^\infty(\Sph,\cg)\cap \cg_\cW=C^\infty(\Sph,\cg)^+$.
Then $g=\exp_G\circ h\in G_\cA^+\sub C^\infty(\Sph,G)^+$.
Likewise for~$g_\ominus$.
Thus $G_\cA^+G_\cA^\ominus$
is an identity neighbourhood in $C^\infty(\Sph,G)$
and hence an open identity neighbourhood,
being an orbit.\\[2.3mm]
Using the projective limit property as a complex manifolds,
the complex analyticity of $\pr^+\colon G^+G^\ominus\to G^+$
follows from that of the corresponding maps
$G_{\cW(m)}^+G_{\cW(m)}^\ominus\to G_{\cW(m)}^+$,
which appear at the bottom in the commutative diagrammes
\[
\begin{array}{rcl}
G^+G^\ominus &\stackrel{\pr^+}{\longrightarrow} & G^+\\
\downarrow\;\; & & \;\;\downarrow\\
G_{\cW(m)}^+G_{\cW(m)}^\ominus &\longrightarrow & G_{\cW(m)}^+.
\end{array}
\]
\end{proof}
%
%
%
%
%
%
%
%
%
%
\section{Loop groups over algebras of germs}
We show that algebras of germs around~$0$
of holomorphic functions or meromorphic functions
on punctured disks
are of limited use for the construction
of Lie groups of loops.\\[2.3mm]
We first discuss the algebra
\[
\cH\, :=\, \dl\,\, \cO(B^{\C}_{1/n}(0)\setminus\{0\},\C)\,\cong\,
\C\{z\}\times \cO(\wh{\C}\setminus\{0\},\C)_*
\]
of germs $[f]$ of holomorphic functions~$f$
on punctured disks around~$0$ in the complex plane;
$\cH$ is endowed with the locally convex direct
limit topology
(cf.\ \cite{Kac}). Then $\cH$ is not a topological
algebra (in the sense of this article):
\begin{la}\label{Hunstet}
The algebra multiplication
$m\colon \cH\times \cH$, $([f],[g])\mto [fg]$
is not continuous.
\end{la}
\begin{proof}
Let $\C\{z\}=\dl \, \cO(B^\C_{1/n}(0),\C)$ be the algebra of germs
of holomorphic functions around~$0$,
endowed with the locally convex direct limit topology.
Let $\C\{z\}'$ be its topological dual spaces, endowed with the topology of
bounded convergence.
Write $f_p\colon \wh{\C}\setminus\{0\}\to\C$
for the principal part of a holomorphic function
$f\colon B^\C_{1/n}(0)\setminus\{0\}\to\C$
and $f_r\colon B^\C_{1/n}(0)\to \C$
for its regular part.
The map
\[
\cH\to
\C\{z\}\times \cO(\wh{\C}\setminus\{0\},\C),\quad
[f]\mto([f_r],f_p)
\]
is well defined and an isomorphism of topological
vector spaces (cf.\ \cite{Kac}).
Moreover, the map
\[
\theta\colon \cO(\wh{\C}\setminus\{0\},\C)_*\to\C\{z\}'
\]
given by $\theta(f)([g]):=\res_0(fg)$
is an isomorphism of topological vector spaces
(cf.\ \cite{Kac}).
The linear map
$\res_0\colon \cH\to\C$, $[f]\mto \res_0(f)$
is continuous, as it is continuous on
each step of the direct system.
If $m$ was continuous, then also the map
\[
h\colon  \cO(\wh{C}\setminus\{0\},\C)\times\C\{z\}\to \cH,\quad
(f,[g])\mto [fg]
\]
would be continuous and hence also $\res_0\circ h$.
But $\res_0\circ f=\ev\circ (\id\times \theta)$
using the identity map of $\cO(\wh{\C}\setminus\{0\},\C)$
and the evaluation map
\[
\ev\colon \C\{z\}'\times\C\{z\}\to \C,\quad (\lambda,[g])\mto\lambda([g]).
\]
Thus $\ev$ would be continuous. But it is well known that
$\ev\colon E'\times E\to\C$ is discontinuous for
each non-normable complex locally convex space~$E$
(cf.\ \cite{KaM}), contradiction.
\end{proof}
We record an immediate consequence:
\begin{la}\label{bra-unstet}
For each
finite-dimensional complex
Lie algebra $(\cg,[.,.]_\cg)$
which is not abelian,
the Lie bracket $[.,.]$ on $\cg_\cH$
is discontinuous.
\end{la}
\begin{proof}
For each $a\in \cg\setminus\{0\}$,
the linear map
\[
\iota_a\colon\cH\to \cg_\cH=\cg\otimes_\C \cH,\quad
[f]\mto a\otimes [f]
\]
is a topological embedding
(cf.\ \ref{proj-tensor}).
We have
\[
[\iota_x([f]),\iota_y([g])]=[x,y]\otimes [f][g]
=\iota_z(m([f],[g]))
\]
for all $[f],[g]\in\cH$.
Hence, if $[.,.]$ was continuous, then also
\[
m=\iota_z^{-1}\circ [.,.]\circ (\iota_x\times \iota_y)
\]
were continuous, which contradicts Lemma~\ref{Hunstet}.
\end{proof}
\begin{rem}
The paper \cite{Her}
claims to obtain complex Lie groups with
Lie algebra $\cg_\cH$,
for certain finite-dimensional
non-abelian Lie algebras~$\cg$.
This is false;
the Lie bracket on $\cg_\cH$
being discontinuous,
$\cg_\cH$ cannot be the Lie algebra
of any Lie group
in the usual sense (as in \cite{Mil}
and the current article).
The construction remains erroneous when
complex analytic Lie groups
in the sense of convenient differential calculus
are considered (as in \cite{KaM}).
In fact,
if one wishes the restriction
$\exp_G|_U$ of the exponential
map to an open $0$-neighbourhood $U\sub\cg$
to be a $\C$-analytic diffeomorphism onto an open set,
one usually accomplishes this by choosing
$U$ small (notably, bounded).\footnote{Only in rare
cases, like $1$-connected nilpotent groups,
we can choose large ($U=\cg$ in the latter case).}
But then the subset $\Omega$ of germs $[f]\in \cH$
having a representative $g\in [f]$ with image
in~$U$ necessitates that $g$ has vanishing principal
part. Hence $\Omega$ has empty interior in~$\cH$
and cannot be used as the domain of a local parametrization
for a Lie group,\footnote{One could only recover the Lie groups of germs of $G$-valued
complex analytic functions on the non-punctured plane~$\C$ around~$0$,
as in \cite{GEM}.} contrary to the attempts in~\cite{Her}.
\end{rem}
The algebra $\cM$ of germs $[f]$
of meromorphic functions around~$0$ in~$\C$
can be turned into a locally convex topological
algebra~\cite{BDS}. The inversion map
$\eta\colon \cM^\times \to\C$
is not continuous (see \cite{BDS}).
Moreover, $\eta$ is not complex
analytic in the sense of convenient differential calculus
(as in \cite{KaM}).
In fact, $\cM$ is a Silva space,
i.e., a locally convex direct limit
\[
\cM=\dl \,E_n
\]
with complex Banach spaces $E_1\sub E_2\sub\cdots$
such that all inclusion maps are compact operators.
Like every Silva space, $\cM$ is compact regular,
i.e., each compact subset $K\sub\cM$
is contained in~$E_n$ and compact in there
for some $n\in\N$
(see, e.g., \cite{Flo} or \cite{GaN}).
Hence
convenient complex analyticity of~$\eta$ is equivalent
to complex analyticity of $\eta|_{\cM^\times\cap E_n}$
for each $n\in\N$,
which in turn is equivalent to
complex analyticity of~$f$
(see \cite{GaN}).
But $f$ is not complex analytic
(being discontinuous).
As a consequence,
$\GL_n(\cM)=(\cM^{n\times n})^\times$
fails to be a complex Lie group
when considered as an open subset of
$\cM^{n\times n}$.\\[2.3mm]
Using coefficients in $\cH$ of~$\cM$,
Lie groups of loops can be constructed
only in limited situation.
\begin{prop}
Let $\cg$ be a finite-dimensional complex
Lie algebra which is nilpotent.
Then the Baker-Campbell-Hausdorff multiplication
makes $\cg_\cM$ a complex Lie group in the sense
of this article,
and it makes $\cg_\cH$ a complex Lie group
in the sense of convenient differential calculus.
\end{prop}
\begin{proof}
Since $\cg_\cH$ and $\cg_\cM$
are nilpotent, the BCH-series is given by a finite sum
and defines a global group structure.
The inversion map is $g\mto -g$ and hence
complex analytic.
Being a linear combination
of nested Lie brackets which are
conveniently complex analytic (resp.,
complex analytic),
the BCH-multiplication
is conveniently complex analytic and
complex analytic, respectively.
\end{proof}
{\small{\bf Helge Gl\"{o}ckner}, Universit\"{a}t Paderborn, Warburger Str.\ 100,
33098 Paderborn,\linebreak
Germany; glockner@math.uni-paderborn.de}\vfill
\end{document}